\documentclass[12pt,english]{amsart}
\usepackage{color}
\usepackage{float}
\usepackage{amstext}
\usepackage{amsmath}
\usepackage{amsthm}
\usepackage{amssymb}
\usepackage{graphicx}
\usepackage{wasysym}
\usepackage{esint}
\usepackage{stmaryrd}
\usepackage{nicematrix}
\usepackage{yhmath}
\usepackage{cancel}

\makeatletter
\numberwithin{equation}{section}
\numberwithin{figure}{section}
\theoremstyle{plain}
\newtheorem{thm}{\protect\theoremname}
\theoremstyle{remark}
\newtheorem{remark}[thm]{\protect\remarkname}
\newtheorem*{rem}{\protect\remarkname}
\theoremstyle{plain}
\newtheorem{prop}[thm]{\protect\propositionname}
\theoremstyle{plain}
\newtheorem{lem}[thm]{\protect\lemmaname}

\usepackage{fancyhdr}
\usepackage{cite}

\NiceMatrixOptions{code-for-first-row = \color{blue},
code-for-first-col = \color{red},
code-for-last-row = \color{green},
code-for-last-col = \color{magenta}}

\newcommand{\cO}{\mathcal{O}}

\newcommand{\abs}[1]{\ensuremath{|#1|}}

\newcommand{\Abs}[1]{\ensuremath{\left|#1\right|}}
\newcommand{\card}{\mathop{\rm Card}}

\renewcommand{\epsilon}{\varepsilon}
\renewcommand{\phi}{\varphi}

\usepackage{babel}

  \providecommand{\lemmaname}{Lemma}
  \providecommand{\propositionname}{Proposition}
  \providecommand{\theoremname}{Theorem}
\providecommand{\theoremname}{Theorem}

\usepackage{babel}

\providecommand{\lemmaname}{Lemma}
\providecommand{\propositionname}{Proposition}
\providecommand{\remarkname}{Remark}
\providecommand{\theoremname}{Theorem}

\usepackage{babel}

\providecommand{\lemmaname}{Lemma}
\providecommand{\propositionname}{Proposition}
\providecommand{\remarkname}{Remark}
\providecommand{\theoremname}{Theorem}

\usepackage{babel}
\providecommand{\lemmaname}{Lemma}
\providecommand{\propositionname}{Proposition}
\providecommand{\remarkname}{Remark}
\providecommand{\theoremname}{Theorem}

\usepackage{babel}
\providecommand{\lemmaname}{Lemma}
\providecommand{\propositionname}{Proposition}
\providecommand{\remarkname}{Remark}
\providecommand{\theoremname}{Theorem}

\makeatother

\usepackage{babel}
\providecommand{\lemmaname}{Lemma}
\providecommand{\propositionname}{Proposition}
\providecommand{\remarkname}{Remark}
\providecommand{\theoremname}{Theorem}

\begin{document}
\title[On the Fourier coefficients]{On the Fourier coefficients of powers\\ of a finite Blaschke product}
\author{Alexander Borichev}
\address{Aix-Marseille Universit\'e, CNRS, Centrale Marseille, I2M, Marseille,
France}
\email{alexander.borichev@math.cnrs.fr}
\author{Karine Fouchet}
\address{Aix-Marseille Universit\'e, Laboratoire ADEF, Campus Universitaire de
Saint-J\'er\^ome, 52 Avenue Escadrille Normandie Niemen, 13013 Marseille, France\vspace*{-0.15cm}}
\address{Aix-Marseille Universit\'e, CNRS, Centrale Marseille, I2M, Marseille,
France}
\email{karine.isambard@univ-amu.fr }
\author{Rachid Zarouf}
\address{Aix-Marseille Universit\'e, Laboratoire ADEF, Campus Universitaire de
Saint-J\'er\^ome, 52 Avenue Escadrille Normandie Niemen, 13013 Marseille, France}
\email{rachid.zarouf@univ-amu.fr}

\keywords{Fourier coefficients, finite Blaschke products, van der Corput lemma, condition number of matrices.}
\subjclass[2020]{30J10, 42A16, 41A60, 15A60.}
\thanks{The work of R.Z. was supported by the pilot center Ampiric, funded by the France 2030 Investment Program operated by the Caisse des D\'ep\^ots.}

\begin{abstract}
Given a finite Blaschke product $B$ we prove asymptotically
sharp estimates on the $\ell^{\infty}$-norm of the sequence of
the Fourier coefficients of $B^{n}$
as $n$ tends to $\infty$. We provide constructive
examples which  
show that our estimates are sharp.
As an application we construct
a sequence of $n\times n$ invertible matrices $T$ with 
arbitrary spectrum in the unit disk and such that the quantity
$|\det{T}|\cdot\|T^{-1}\|\cdot\|T\|^{1-n}$ grows as
a power of $n$. This is motivated by Sch\"affer's question
on norms of inverses.
\end{abstract}

\maketitle

\section{Introduction}

Let $\mathbb{D}$ be the unit disk of the complex plane, $m\ge1$, let $\sigma=\left(\lambda_{1},\dots,\lambda_{m}\right)\in\mathbb{D}^{m}$
and let $B$ be the finite Blaschke product of degree $m\ge1$ associated
to $\sigma$: 
\[
B(z)=\prod_{j=1}^{m}b_{\lambda_j}(z),
\]
where $b_{\lambda}(z)=\frac{z-\lambda}{1-\bar{\lambda}z}$ is the
Blaschke factor corresponding to $\lambda\in\mathbb{D}$. 
Here and later on we omit the standard unimodular factor $\bar\lambda/|\lambda|$ which is of no importance for the questions we study here. Furthermore, we assume that $0\not\in\sigma$. 
For $k\ge0$
we consider the $k^{{\rm th}}$-Fourier coefficient of $B^{n}$ defined
by: 
\[
\widehat{B^{n}}(k)=\frac{1}{2\pi}\int_{0}^{2\pi}B^n(e^{{\rm i}\theta})e^{-\rm ik\theta}\,d\theta.
\]
We study the asymptotic behavior of the $\ell^{\infty}$-norm
of the sequence $\bigl(\widehat{B^{n}}(k)\bigr)_{k\ge0}$:
\[
\|\widehat{B^{n}}\|_{\ell^{\infty}}:=\sup_{k\ge0}\abs{\widehat{B^{n}}(k)},
\]
 as $n$ grows large.{} We use some standard notation
from asymptotic analysis. For two positive functions
$f,g$ we say that $f$ is
dominated by $g$, denoted by $f\lesssim g$, if there is a constant
$c>0$ such that $f\le cg$. We say that $f$ and $g$ are comparable,
denoted by $f\asymp g$, if both $f\lesssim g$ and $g\lesssim f$.
\\The case $B=b_{\lambda}$, $\lambda$ being arbitrary in $\mathbb{D}\setminus\{0\}$,
is well-studied \cite{BCP,SZ,BFZ,Mey} and it is known that 
\begin{equation}
\|\widehat{b_{\lambda}^{n}}\|_{\ell^{\infty}}\asymp n^{-1/3}.\label{eq:basic}
\end{equation}
We mention in passing that Y.~Meyer recently rediscovered the upper bound in \eqref{eq:basic} in view of applying it to some  sparse crystalline measure constructions \cite{Mey}.
Precise asymptotic formulas for the $k^{{\rm th}}$ Fourier
coefficients of $b_{\lambda}^{n}$, $k\in[0,\infty)$ as $n\rightarrow\infty$
have been recently obtained in \cite{BFZ} where the authors distinguish
several regions of different asymptotic behavior of $\widehat{b_{\lambda}^{n}}(k)$
in terms of $k$ and $n.$ These asymptotic formulas are applied to
the construction of strongly annular functions with Taylor coefficients
satisfying sharp summation properties, which improves and generalizes
the results in \cite{BCP}. \\
Information about those coefficients has also recently
been exploited in operator theory~\cite{LLQR} to identify decreasing
weights for which composition operators are bounded on weighted Hardy
spaces of Hilbert type. Observe also that a connection was earlier
established in \cite{BS} between the asymptotics of the Fourier coefficients
of $B^{n}$ and the boundedness of the composition operator $\mathcal C_B$, $\mathcal C_B(f)=f\circ B$ 
on the analytic Beurling--Sobolev space of analytic functions
in $\mathbb{D}$ whose sequence of Taylor coefficients belongs to
$l^{p},$ $p\in[1,\infty]$. Generally speaking, to verify whether
$\mathcal C_B$ is a bounded linear operator from one Banach space $X$
of analytic functions into another, say $Y$, it is often enough to
know the asymptotic behavior of $\|B^{n}\|_Y$. It is shown in
particular in \cite{BS} that~for any finite Blaschke product $B$
\begin{equation}
\|\widehat{B^{n}}\|_{\ell^{p}}\asymp n^{\frac{2-p}{2p}}\ \textnormal{for}\ p\in[1,2]\label{eq:BS}
\end{equation}
and that
\begin{equation}
\|\widehat{B^{n}}\|_{\ell^{p}}\gtrsim n^{\frac{2-p}{2p}}\ \textnormal{for}\ p\in[2,\infty].\label{eq:BS_Weak}
\end{equation}
The study of the case $p=1$ in \eqref{eq:BS} with a single Blaschke factor $B=b_{\lambda}$
was probably initiated by J.-P.~Kahane~\cite{Kah}. Applying van
der Corput type estimates on $\widehat{b_{\lambda}^{n}}(k)$ \cite[p. 253]{Kah},
he proved that $\|\widehat{b_{\lambda}^{n}}\|_{\ell^{1}}\asymp n^{-1/2}$
which is in line with \eqref{eq:BS}. Notice that his motivation \cite[Theorem 1]{Kah}
was different from the above: he generalized a theorem by Z.~K.~Leibenson
\cite{Leib}, which is a special case of a theorem~\cite[Theorem 4.1.3]{Rud}
about homomorphisms of group algebras due to P.~T.~Cohen. The exact
asymptotic $n$-dependency of $\|\widehat{b_{\lambda}^{n}}\|_{\ell^{p}}$
for $p\in[1,\infty]$ was described in~\cite{SZ} and a change of asymptotic behavior 
at $p=4$ was in particular discovered \cite{SZ,BFZ}. In this paper
we focus on the case $p=\infty$: in this case the estimate \eqref{eq:BS_Weak}
reads 
\[
\|\widehat{B^{n}}\|_{\ell^{\infty}}\gtrsim n^{-1/2}
\]
for any finite Blaschke product $B$. We will prove that $\|\widehat{B^{n}}\|_{\ell^{\infty}}$
is never of order $n^{-1/2}$. More precisely, for any finite Blaschke product
$B$, there exists an integer $N\ge3$ such that
\begin{equation}
\|\widehat{B^{n}}\|_{\ell^{\infty}}\asymp n^{-1/N}.\label{eq:main_estimate}
\end{equation}
For every integer $N\ge3$ we will exhibit an example of a Blaschke
product $B$ of degree $N$ satisfying estimate \eqref{eq:main_estimate}.
It is known \cite{SZ} that for $N=3$ every Blaschke product $B(z)\not=z$ of degree
1 satisfies \eqref{eq:main_estimate}. For $N=5$ -- respectively
$N=7$ -- we produce explicit examples of Blaschke products
of degree 2 -- respectively degree 4 -- such that \eqref{eq:main_estimate}
holds. \\
Finally we apply \eqref{eq:main_estimate} to construct
a new class of $n\times n$ invertible matrices $T$ with arbitrary spectrum 
in $\mathbb{D}$ such that 
\begin{equation}
\frac{|\det{T}|\cdot\|T^{-1}\|}{\|T\|^{n-1}}\gtrsim n^{1/N}\label{eq:schaffer_arb_spec}
\end{equation}
for some integer $N\ge3.$ This is motivated by Sch\"affer's question on the
norms of $n\times n$ invertible matrices, see Section 3. 

\subsection*{Outline of the paper}

The paper is organized as follows. In Section 2 we establish sharp estimates on the $\ell^{\infty}$-norm of the sequence $(\widehat{B^{n}}(k))_{k\ge0}$
for large $n$, where $B$ is an arbitrary Blaschke product, see Theorem \ref{thm:upper_bd}.
Concrete examples of finite Blaschke products achieving 
estimate \eqref{eq:main_estimate} are provided in Theorem \ref{thm:constr_examples}.
In Section 3 we state Theorem \ref{prop_phi}, which extends the work \cite{SZ,SZ1,SZ2} already undertaken to answer a question raised by Sch\"affer on norms of inverse matrices. In Theorem \ref{prop_phi} we exhibit a sequence of $n\times n$ matrices $T$ satisfying \eqref{eq:schaffer_arb_spec}, which we obtain as an application of the techniques used to prove Theorem~\ref{thm:upper_bd}. Section 4 is devoted to the proof
of Theorem \ref{thm:upper_bd} which combines the use of two van der Corput lemmata with a stationary phase type argument. In Section 5 we prove Theorem \ref{thm:constr_examples}
and in particular we construct finite Blaschke products $B$ of degree
2 -- respectively 4 -- such that $\|\widehat{B^{n}}\|_{\ell^\infty}$
grows as $n^{-1/5}$ -- respectively $n^{-1/7}$. Finally in Section 6 we prove Theorem~\ref{prop_phi} combining a duality method with 
the techniques that lead to the proof of Theorem~\ref{thm:upper_bd}. An alternative proof of Theorem~\ref{prop_phi} using Theorem~\ref{thm:upper_bd} is provided in Remark~\ref{rem9}.

\section{Notation and statement of our results}

In this section we give sharp asymptotics for the $\ell^{\infty}$-norm of the sequence $\bigl(\widehat{B^{n}}(k)\bigr)_{k\ge0}$,
see Theorem \ref{thm:upper_bd} below, and constructive examples that
achieve these asymptotics, see Theorem \ref{thm:constr_examples} below.
We denote by $\psi_{B}(\theta)$ the continuous argument (determined modulo $2\pi$)
of $B(e^{{\rm i}\theta})$:
$$
B(e^{{\rm i}\theta})=\exp\left({\rm i}\psi_{B}(\theta)\right)
$$
and by $(\xi_\ell)_{\ell=1}^s$ the sequence of (consecutive) zeros
of $\psi_{B}''$ on $[0,2\pi)$ with respective multiplicities $(N_\ell-2)_{\ell=1}^s$, 
$N_\ell\ge 3$. This means that 
\[
0\le \xi_{1}<\xi_{2}<\ldots<\xi_{s}<2\pi
\]
and that for any $\ell=1,\ldots, s$ 
\[
\psi''_{B}(\xi_\ell)=\ldots=\psi_{B}^{(N_\ell-1)}(\xi_\ell)=0,\qquad\psi_{B}^{(N_\ell)}(\xi_\ell)\not=0.
\]
Without loss of generality (using a simple rotation if necessary), we can assume that $\xi_1>0$.

\begin{thm}
\label{thm:upper_bd} Let $B$ be a finite Blaschke product,
and let $\psi_{B}$, $(\xi_\ell)_{\ell=1}^{s}$, $(N_\ell)_{\ell=1}^{s}$ be defined as
above. Then we have 
$$
\|\widehat{B^n}\|_{\ell^\infty}\asymp n^{-1/N},
$$
where $N=\max_{1\le \ell\le s}N_\ell$. 
\end{thm}

It is known \cite{BFZ} that if $B=b_{\lambda}$ is the
Blaschke factor associated to any fixed $\lambda\in\mathbb{D}\setminus\{0\},$
then $N=3$ and 
\[
\|\widehat{B^{n}}\|_{\ell^{\infty}}\asymp n^{-1/3},
\]
which proves the sharpness of Theorem \ref{thm:upper_bd} for $N=3$.
Here we establish the following result. 

\begin{thm}
\label{thm:constr_examples} For every $N\ge3$, there exists a Blaschke
product $B$ of degree $N$ such that 
\begin{equation}
\|\widehat{B^{n}}\|_{\ell^\infty}\asymp n^{-1/N}.\label{B}
\end{equation}
Moreover:\\
{\rm(1)} For $N=3$ every Blaschke product of degree 1, different from
identity, satisfies \eqref{B}. \\
{\rm(2)} For $N=5$ there exists a Blaschke product of degree 2 such that
\eqref{B} holds. \\
{\rm(3)} For $N=7$ there exists a Blaschke product of degree 4 such that
\eqref{B} holds. 
\end{thm}

\begin{rem}
The following two questions remain open:\\
(1) Every nontrivial Blaschke product $B$ of
degree $1$ satisfies \eqref{B} with $N=3$. In the case $N=5$, a Blaschke
product of degree $2$ satisfies \eqref{B}, and in the case $N=7$, a Blaschke product
of degree $4$ satisfies \eqref{B}. 
What is the minimal degree of a Blaschke product $B_{N}$ such that \eqref{B} holds for a fixed $N\ge 6$?\\
(2) Is there an infinite Blaschke product $B$ such that for $n\ge1$
we have 
\[
\|\widehat{B^{n}}\|_{\ell^\infty}\asymp\frac{1}{\log(n+1)}\,?
\]
\end{rem}

\section{Application to operator theory}

A well-established circle of questions in Operator Theory concerns estimating the norm of the inverse of an invertible $n\times n$ matrix $T$ on an $n$-dimensional Banach space $X$. 
In the early 1970's, B.~L.~Van der Waerden, W.~A.~Coppel and J.~J.~Sch\"affer \cite{SJ} studied the smallest quantity $C$, denoted $C_{\mathcal{B}}(n)$,
such that 
\[
|\det{T}|\cdot\|T^{-1}\|\le C \|T\|^{n-1}
\]
for any invertible $T$ and for any $X$. 
In 1970, Sch\"affer \cite{SJ} proved the estimate $C_{\mathcal{B}}(n)\le\sqrt{en}$,
and showed that the inequality 
\begin{equation}
|\det T|\cdot\|T^{-1}\|\le 2\|T\|^{n-1}\label{eq:l1linfty-1}
\end{equation}
holds for any invertible $T$ acting on $\mathbb{C}^{n}$ endowed
with the $\ell^{1}$-norm or with the $\ell^{\infty}$-norm (and that
\eqref{eq:l1linfty-1} is optimal in both cases). This led him to
make the conjecture -- nowadays known as \textit{Sch\"affer's
conjecture}{} -- that $C_{\mathcal{B}}(n)=2$, $n\in\mathbb{N}$.
The latter was disproved first by E.~Gluskin, M.~Meyer and A.~Pajor
\cite{GMP}, J.~Bourgain \cite{GMP} \footnote{The same article \cite{GMP} contains an appendix
with a stronger estimate due to Bourgain.}{} and later by H.~Queff\'elec \cite{Que} who also proved
that $C_{\mathcal{B}}(n)\gtrsim\sqrt{n}$. The above mentioned results
make use of the following analytic expression for $C_{\mathcal{B}}(n)$,
given in \cite{GMP}, in terms of a ``max-min-type'' 
optimization problem : 
\[
C_{\mathcal{B}}(n)=\sup_{(\lambda_1,\ldots,\lambda_n)\in\mathbb{D}^{n}}\Phi(\lambda_1,\ldots,\lambda_n),
\]
where 
\begin{multline*}
\Phi(\lambda_1,\ldots,\lambda_n)\\:=\inf\biggl\{ \sum_{k=1}^{\infty}\abs{a_k}: f(z)=\prod_{j=1}^{n}\lambda_j+\sum_{k=1}^{\infty}a_kz^k,\,f(\lambda_j)=0,\,j=1,\ldots, n\biggr\} .
\end{multline*}
The proof by Queff\'elec makes use of Bourgain's lower bound on $\Phi$ \cite[Inequality (2.2)]{SZ2}, 
which he combines with a number theoretic argument to prove the existence
of a family $(\lambda_1,\ldots,\lambda_n)$ on the circle of
radius $1-1/n$ such that $\Phi(\lambda_1,\ldots,\lambda_n)\gtrsim\sqrt{n}$.
\\The question of finding a concrete exemple of a sequence $(\lambda_1,\ldots,\lambda_n)$
in $\mathbb{D}^n$ such that $\Phi(\lambda_1,\ldots,\lambda_n)\gtrsim\sqrt{n}$ 
has been recently addressed by O.~Szehr--R.~Zarouf \cite{SZ2} who
proved that for every $\lambda\in\mathbb{D}\setminus\{0\}$ we have $\Phi(\lambda,\ldots,\lambda)\gtrsim\sqrt{n}$, 
where $\lambda$ is repeated $n$ times. 
As a consequence, they construct in \cite{SZ2} an explicit class
of counterexamples to Sch\"affer's conjecture: a sequence of invertible lower triangular  
Toeplitz matrices $T_{\lambda}\in\mathcal{M}_{n}$ with singleton
spectrum $\{\lambda\}\subset\mathbb{D}\backslash\{0\}$ such that
\[
\abs{\lambda}^{n}\|T_\lambda^{-1}\|\ge c(\lambda)\sqrt{n}\|T_\lambda\|^{n-1},
\]
where $c(\lambda)>0$ depends only on $\lambda$. 

Furthermore, Gluskin--Meyer--Pajor \cite[p.~2, lines 25--26]{GMP}
mention that it is of interest to find concrete examples $(\lambda_1,\ldots,\lambda_n)$
for which $\Phi(\lambda_1,\ldots,\lambda_n)$ grows. We will
combine the approach in \cite{SZ2} with Theorem \ref{thm:upper_bd}
to prove that given a fixed $m\ge1$ and an arbitrary sequence $(\lambda_1,\lambda_2,\ldots,\lambda_m)\in\mathbb{D}^{m}$, 
the sequence 
$$
(\lambda_1,\ldots,\lambda_1,\lambda_2,\ldots,\lambda_2,\ldots,\lambda_m,\ldots,\lambda_m)\in\mathbb{D}^{n\times m}
$$
(i.e. each $\lambda_{i}$ is repeated according to its multiplicity
$n$) satisfies the estimate 
\[
\Phi(\lambda_1,\ldots,\lambda_1,\lambda_2,\ldots,\lambda_2,\ldots,\lambda_m,\ldots,\lambda_m)\gtrsim n^{1/N}
\]
 for some integer $N\ge3$, as $n$ tends to infinity. As a consequence, we obtain a family of invertible lower triangular 
matrices $T\in\mathcal{M}_{nm}$ with arbitrary spectrum $(\lambda_1,\lambda_2,\ldots,\lambda_m)$ 
such that 
\[
|\det T|\cdot\|T^{-1}\|\gtrsim n^{1/N}\|T\|^{n-1}
\]
 for some integer $N\ge3$, as $n$ tends to infinity. Our construction of $T$ is given in Section~\ref{Schaeffer}. 

\subsection{Background}

We denote by $H(\mathbb{D})$ the space of the functions analytic in the
open unit disk $\mathbb{D}$, 
and by $H^{\infty}$ the Banach algebra of bounded analytic functions
in $\mathbb{D}$, endowed with the supremum norm. 
We recall that the Hardy space $H^{2}$ is defined as the
subspace of $H(\mathbb{D})$ consisting of the functions $f$ such
that 
\[
\Vert f\Vert_{H^{2}}^{2}:=\sup_{0\le r<1}\int_{\mathbb{T}}\left|f(rz)\right|^{2}\,d\nu(z)<\infty,
\]
where $\nu$ is the normalized Lebesgue measure on the unit circle $\mathbb{T}:=\left\{ z\in\mathbb{C}:\,|z|=1\right\} $.
Let $\sigma$ be a finite sequence of points in $\mathbb{D}.$ The
finite Blaschke product $B=B_{\sigma}$ corresponding to $\sigma$
is defined by 
\[
B=B_{\sigma}=\prod_{\lambda\in\sigma}b_{\lambda},
\]
where $b_{\lambda}=\frac{z-\lambda}{1-\overline{\lambda}z}$ is the
Blaschke factor corresponding to $\lambda\in\mathbb{D}$. Then one
defines the model space $K_{B}$ as the finite dimensional subspace
of $H^{2}$ given by 
\[
K_B:=\left(BH^{2}\right)^{\perp}=H^{2}\ominus BH^{2}.
\]
Let $\sigma
=(\lambda_{1},\dots,\lambda_{q})\in\mathbb{D}^q$. We set $f_k=\dfrac{1}{1-\overline{\lambda_{k}}z}$, $k=1,\ldots,q$.
Observe that $\|f_k\|_{H^{2}}=\left(1-\vert\lambda_k\vert^{2}\right)^{-1/2}$. The family $(f_k)_{1\le k\le n}$ is a basis of $K_B$.
Furthermore, the family $(e_k)_{1\le k\le n}$
given by 
\begin{equation*}
e_1=\frac{f_1}{\|f_1\|_{H^2}}\,\quad\mbox{and}\quad e_{k}=\frac{f_{k}}{\|f_k\|_{H^2}}{\displaystyle \prod_{j=1}^{k-1}}b_{\lambda_{j}},\quad k=2,\ldots,q, 
\end{equation*}
is an orthonormal basis of $K_B$ (known as the Malmquist--Walsh
basis, see \cite{NN}). 

The backward shift operator $S: f\mapsto (f-f(0))/z$ acts on $K_B$.

\subsection{Constructive counterexamples to Sch\"affer's conjecture}  \label{Schaeffer}
\phantom{A} We put 
$$
\sigma
=(\lambda_{1},\dots,\lambda_{1},\lambda_{2},\dots,\lambda_{2},\dots,\lambda_{m},\dots,\lambda_{m})\in\mathbb{D}^{n m}, 
$$
where distinct $\lambda_i$ are arbitrary in $\mathbb{D}\backslash\{0\}$ and are repeated $n$ times. 
We consider the Malmquist--Walsh basis of $K_B$ given by 
$$
e_{sn+t}=\biggl(\prod_{j=1}^sb^n_{\lambda_j} \biggr)\frac{f_{s+1}}{\|f_{s+1}\|_{H^2}}b_{\lambda_{s+1}}^{t-1},\qquad 0\le s<m,\,1\le t\le n.
$$
The backward shift operator $S$ has a lower triangular matrix with respect to this basis, with diagonal   
$(\lambda_{1},\dots,\lambda_{1},\dots,\lambda_{m},\dots,\lambda_{m})$, see \cite{SO} for the entry-wise description of this matrix. 
Hence, 
\begin{equation}  \label{det1}
\det S= \prod_{j=1}^m\lambda^n_j.
\end{equation}
We will use the duality method from \cite{SZ2} and combine
it with our upper estimates on the Fourier coefficients of the $n^{\text{th}}$
power of a finite Blaschke product $B$ (see Theorem \ref{thm:upper_bd})
to show that there exists a Banach space norm on $K_B$ 
such that the operator $S$ with spectrum $(\lambda_1,\lambda_2,\ldots,\lambda_m)$ acting 
there, satisfies the asymptotic relation 
$|\det S|\cdot \|S^{-1} \|\cdot \| S \|^{1-n} \gtrsim n^{1/N}$. 

Following \cite{GMP}, we choose the norm $\|\cdot\|_{\ell^\infty_A}$ on $K_B$:
$$
\|f\|_{\ell^\infty_A}=\|\widehat{f}\|_{\ell^\infty}.
$$
Then $\|S\|_{\ell^\infty_A\to \ell^\infty_A}\le 1$. Furthermore, if $\dim K_B>1$, then $\|S\|_{\ell^\infty_A\to \ell^\infty_A}=1$. 
We are now ready to state the corresponding result. 

\begin{thm} \label{prop_phi}
\label{new} Let 
$$
\sigma=(\lambda_1,\ldots,\lambda_1,\lambda_2,\ldots,\lambda_2,\ldots,\lambda_m,\ldots,\lambda_m)\in\mathbb{D}^{n m},
$$
where distinct $\lambda_i$ are arbitrary in $\mathbb{D}\backslash\{0\}$
and are repeated $n$ times. There exists an
integer $N \ge 3$ such that 
\begin{enumerate}
\item[(i)] $|\det S|\cdot \|S^{-1}\|_{\ell^\infty_A\to \ell^\infty_A}
\gtrsim {n^{1/N}}$
\\ and
\item[(ii)]  $\Phi(\lambda_1, \ldots, \lambda_1, \ldots, \lambda_m,\ldots, \lambda_m) \gtrsim n^{1/N} $.
\end{enumerate}
\end{thm}

Here $N$ is the integer associated with the Blaschke product $B= \prod_{j=1}^m b_{\lambda_j}$ in Theorem~\ref{thm:upper_bd}.

\section{\label{sec:Proof-Theorem1}Proof of Theorem \ref{thm:upper_bd} }

The proof of Theorem~\ref{thm:upper_bd} makes use of the following
van der Corput lemmata, see, for example, \cite[Lemma 2.1 and Lemma 2.2, p. 56]{Ivi}. 

\begin{lem}\label{VDCP_order_1} Let $g$ be a real function continuously
differentiable on the interval $[a,b]\subset\mathbb{R}$ such that
$g$ and $g'$ are monotone and $g'$ does not vanish on $[a,b].$
Then 
\[
\Abs{\int_{a}^{b}e^{{\rm i}g(t)}\,dt}\le\frac{2}{\abs{g'(a)}}+\frac{2}{\abs{g'(b)}}.
\]
\end{lem}

\begin{lem}
\label{lem:VDCP} Let $F(x)$ be a real, twice differentiable function
in $[a,b]$ such that $F''(x)\ge\mu>0$ or $F''(x)\le-\mu<0$. Let
$G(x)$ be a positive monotonic function in $[a,b]$ such that $G(x)\le M$.
Then 
\[
\Abs{\int_{a}^{b}G(x)e^{{\rm i}F(x)}\,dx}\le\frac{8M}{\sqrt{\mu}}.
\]
\end{lem}

\begin{proof}[Proof of Theorem \ref{thm:upper_bd}.]


For $j=1,\ldots,m$ we write $\lambda_j=\rho_{j}\exp({\rm i}\theta_j)$, where $\rho_{j}\in(0,1)$
and $\theta_{j}\in(-\pi,\pi]$. A direct computation shows that 
\[
\psi'_{B}(\theta)=|B'(e^{{\rm i}\theta})|=\Bigl|\frac{B'(e^{{\rm i}\theta})}{B(e^{{\rm i}\theta})}\Bigr|=\biggl|\sum_{1\le j\le m}\frac{b'_{\lambda_j}
(e^{{\rm i}\theta})}{b_{\lambda_j}(e^{{\rm i}\theta})}\biggr|.
\]
Since 
\[
\frac{zb'_{\lambda}(z)}{b_{\lambda}(z)}=\Re\frac{1+\overline{\lambda}z}{1-\overline{\lambda}z}\ge0,\qquad\lambda\in\mathbb{D},\,z\in\mathbb{T},
\]
we have 
\begin{equation}
\psi'_{B}(\theta)=\Re\sum_{1\le j\le m}\frac{1+\overline{\lambda_{j}}e^{{\rm i}\theta}}{1-\overline{\lambda_j}e^{{\rm i}\theta}}.\label{eq1}
\end{equation}

Therefore, 
\[
\psi_{B}'(\theta)=\sum_{j=1}^{m}\frac{1-\rho_{j}^{2}}{1+\rho_{j}^{2}-2\rho_{j}\cos\left(\theta-\theta_{j}\right)}.
\]
The $k^{th}$-Fourier coefficient of $B^{n}$ is given by 
\[
\widehat{B^{n}}(k)=\frac{1}{2\pi}\int_{0}^{2\pi}\exp\left({\rm i}\left(n\psi_{B}(\theta)-k\theta\right)\right)\,d\theta.
\]
We put 
\[
f(\theta)=\psi_{B}(\theta)-\frac{k}{n}\theta, 
\]
and given $\epsilon>0$ we write (recall that $\xi_{1}>0$) 
\begin{multline*}
\int_{0}^{2\pi}\exp\left({\rm i}n\, f(\theta)\right){\rm d}\theta  =\sum_{\ell=1}^{s}\int_{\xi_{\ell}-\epsilon}^{\xi_{\ell}+\epsilon}\exp\left({\rm i}n\,f(\theta)\right)\,d\theta
  \\+\sum_{\ell=1}^{s-1}\int_{\xi_{\ell}+\epsilon}^{\xi_{\ell+1}-\epsilon}\exp\left({\rm i}n\,f(\theta)\right)\,d\theta
  +\int_{0}^{\xi_{1}-\epsilon}\exp\left({\rm i}n\,f(\theta)\right)\,d\theta\\
  +\int_{\xi_{s}+\epsilon}^{2\pi}\exp\left({\rm i}n\,f(\theta)\right)\,d\theta.
\end{multline*}
We put 
\[
\epsilon=n^{-1/N}\rightarrow0,\qquad n\rightarrow\infty.
\]
Clearly 
\[
\biggl|\int_{\xi_{\ell}-\epsilon}^{\xi_{\ell}+\epsilon}\exp\left({\rm i}n\,f(\theta)\right)\,d\theta\biggr|\le2\epsilon,
\]
and therefore the first sum is bounded from above by  
\[
\biggl|\sum_{\ell=1}^{s}\int_{\xi_{\ell}-\epsilon}^{\xi_{\ell}+\epsilon}\exp\left({\rm i}n\,f(\theta)\right)\,d\theta\biggr|\lesssim n^{-1/N}.
\]

Writing the Taylor expansion (of order $N_{\ell}-2$) of $\psi_{B}''$
near $\xi_{\ell}$ we obtain that 
\[
\psi_{B}''(\theta)=\frac{\psi_{B}^{(N_{\ell})}(\xi_{\ell})}{(N_{\ell}-2)!}\left(\theta-\xi_{\ell}\right)^{N_{\ell}-2}\left(1+\cO\left(\theta-\xi_{\ell}\right)\right),\qquad\theta\to\xi_{\ell}.
\]
Therefore, if $\theta$ is close to $\xi_{\ell}$ but satisfies $\abs{\theta-\xi_{\ell}}\ge\epsilon$,
then 
\begin{equation}
\abs{\psi_{B}''(\theta)}\gtrsim n^{-(N_{l}-2)/N}\label{eq2}.
\end{equation}

Since $\psi_{B}''$ does not vanish on the intervals $(0,\xi_{1}-\epsilon)$, 
$(\xi_{s}+\epsilon,2\pi)$ and $(\xi_{\ell}+\epsilon,\xi_{\ell+1}-\epsilon)$
for $\ell=1,\ldots,s-1$, we obtain that \eqref{eq2} holds everywhere
outside of the intervals $[\xi_{\ell}-\epsilon,\xi_{\ell}+\epsilon]$. Now
we can apply Lemma~\ref{lem:VDCP} to the remaining integrals 
\begin{equation*}
\int_{\xi_{\ell}+\epsilon}^{\xi_{\ell+1}-\epsilon}e^{{\rm i}n\,f(\theta)}\,d\theta,
\qquad
\int_{0}^{\xi_{1}-\epsilon}e^{{\rm i}n\,f(\theta)}\,d\theta,\qquad
\int_{\xi_{s}+\epsilon}^{2\pi}e^{{\rm i}n\,f(\theta)}\,d\theta.
\end{equation*}
Since 
\[
|nf''|\gtrsim n^{2/N}
\]
on these intervals, these integrals are also $\cO\left(n^{-1/N}\right)$.
Thus,
$$
\|\widehat{B^n}\|_{\ell^\infty}\lesssim n^{-1/N}.
$$

Let $r\in\{1,\ldots,s\}$ be such that $N_r=N$ and define  
$$
\mathcal L=\bigl\{\ell:1\le \ell\le s,\, N_\ell=N,\, \psi_{B}'(\xi_\ell)=\psi_{B}'(\xi_r)\bigr\},\quad D=\card \mathcal L.
$$

We will show that the maximal value of $\abs{\widehat{B^{n}}(k)}$ for $k\ge0$ 
is (asymptotically as $n$ grows large) attained at $k=k_d=k_d(n)=[n\psi_{B}'(\xi_r)]+d$, $0\le d<D$ 
(where $\left[A\right]$ means the integer part of $A$), that is,
\[
\max_{0\le d<D}|\widehat{B^{n}}(k_d(n))|\asymp n^{-1/N}.
\]
Recall that the $k_d^{th}$-Fourier coefficient of $B^{n}$ is given
by 
$$ 
\widehat{B^n}(k_d) =\frac{1}{2\pi}\int_{0}^{2\pi}\exp\left({\rm i}n\,f_d(\theta)\right)\,d\theta,
$$ 
where  
\[
f_d(\theta)=\psi_{B}(\theta)-\frac{k_d}{n}\theta.
\]
Let $(\epsilon_{\ell})_{\ell=1}^{s}$ be a sequence of nonnegative numbers.
Each $\epsilon_\ell$ for $\ell=1,\ldots, s$ will be chosen below over the
proof, depending on the nature of $\xi_{\ell}$. 
Let $\tau_0=0$, $\tau_\ell\in(\xi_\ell+\epsilon_{\ell},\xi_{\ell+1}-\epsilon_{\ell+1})$, $1\le\ell\le s-1$, $\tau_s=2\pi$, 
$J_\ell=[\xi_\ell-\epsilon_{\ell},\xi_{\ell}+\epsilon_{\ell}]$, $1\le\ell\le s$. Then 
$$
[0,2\pi)\setminus\bigsqcup_{1\le\ell\le s}J_\ell=\bigsqcup_{1\le\ell\le s}[\tau_{\ell-1},\xi_\ell-\epsilon_{\ell})\sqcup (\xi_{\ell}+\epsilon_{\ell},\tau_\ell)=:\bigsqcup_{1\le\ell\le s}(J'_\ell\sqcup J''_\ell).
$$ 

We have 
$$
\int_{0}^{2\pi}e^{{\rm i}n\,f(\theta)}\,d\theta  =\sum_{1\le\ell\le s}\int_{J_\ell}e^{{\rm i}n\,f(\theta)}\,d\theta+\sum_{1\le\ell\le s}
\int_{J'_\ell\sqcup J''_\ell}e^{{\rm i}n\,f(\theta)}\,d\theta.
$$
We divide our argument into three steps in order to show that the main
contribution of $\int_{0}^{2\pi}\exp\left({\rm i}n\,f(\theta)\right)\,d\theta$
is due to $J_\ell$, $\ell\in\mathcal L$. 

The content of Step
1 is computing an asymptotic formula for $\Sigma=\Sigma_d$, the sum of the integrals
\[
I_{\ell,d}=\int_{J_\ell}\exp\left({\rm i}n\,f_d(\theta)\right)\,d\theta,\qquad \ell\in\mathcal L,
\]
with $0\le d<D$, for suitable choices of $\epsilon_\ell$. In Step 2 we estimate from
above the integrals
\[
\int_{J_\ell}\exp\left({\rm i}n\,f(\theta)\right)\,d\theta,
\]
for $1\le \ell\le s$, $\ell\not\in\mathcal L$, and show that they are asymptotically
much smaller than $\Omega=\max_{0\le d<D}|\Sigma_d|$, again for suitable choices of $\epsilon_{\ell}$.
Finally the goal of Step 3 is to show -- in the same spirit as in
Step 2 -- that the integrals 
\[
\int_{J'_\ell}\exp\left({\rm i}n\,f(\theta)\right)\,d\theta,\qquad \int_{J''_\ell}\exp\left({\rm i}n\,f(\theta)\right)\,d\theta,
\]
are also asymptotically much smaller than $\Omega$. \\

Step 1. To make the notation less cluttered we set $\xi=\xi_\ell$, $F=F_d=nf_d$, $\epsilon=\epsilon_\ell$, so that 
\[
I=I_\ell=I_{\ell,d}=\int_{\xi-\epsilon}^{\xi+\epsilon}\exp\left({\rm i}F_d(\theta)\right)\,d\theta.
\]
Without loss of generality we may assume that 
\[
\psi_B^{(N)}(\xi)>0.
\]
In the opposite case the argument is analogous.\\

Furthermore, we fix $\delta\in(0,1/(2N^2))$ and set 
\begin{equation}
\epsilon=n^{\delta-\frac1N}.
\label{st1}
\end{equation}
This choice of $\epsilon$ is motivated below in the proofs of formulas \eqref{eq:I_N_even} and \eqref{eq:I_N_odd}. \\

We will first establish that:

(a) If $N$ is even, then 
\begin{multline}
I_{\ell,d}=\frac{2}{N}\left(\frac{N!}{n\psi_{B}^{(N)}(\xi)}\right)^{1/N}\times \\ \times \exp\left({\rm i}n\psi_{B}(\xi)-{\rm i}k_d\xi+\frac{{\rm i}\pi}{2N}\right)\Gamma(1/N)+o\left(\frac{1}{n^{1/N}}\right).\label{eq:I_N_even}
\end{multline}

(b) If $N$ is odd, then 
\begin{multline}
I_{\ell,d}=\frac{2}{N}\left(\frac{N!}{n\psi_{B}^{(N)}(\xi)}\right)^{1/N}\times \\ \times\exp\left({\rm i}n\psi_{B}(\xi)-{\rm i}k_d\xi\right)\cos\left(\frac{\pi}{2N}\right)\Gamma(1/N)+o\left(\frac{1}{n^{1/N}}\right).\label{eq:I_N_odd}
\end{multline}

Here and later on, $\Gamma$ is the Gamma function.

Writing the Taylor expansion of $\psi_{B}$ near
$\xi$ we obtain that
\begin{multline*}
\psi_{B}(\theta)=\psi_{B}(\xi)+(\theta-\xi)\psi_{B}'(\xi)\\+\frac{(\theta-\xi)^{N}}{N!}\psi_{B}^{(N)}(\xi)+\frac{(\theta-\xi)^{N+1}}{(N+1)!}\psi_{B}^{(N+1)}(\xi+t_\theta(\theta-\xi)),
\end{multline*}
for some $t_\theta\in(0,1)$.
Therefore, 
\begin{multline*}
\psi_{B}(\theta)-\frac{k}{n}\theta  =\psi_{B}(\xi)-\frac{k}{n}\xi+(\theta-\xi)\left(\psi_{B}'(\xi)-\frac{k}{n}\right)+\frac{(\theta-\xi)^{N}}{N!}\psi_{B}^{(N)}(\xi)\\
 +\frac{(\theta-\xi)^{N+1}}{(N+1)!}\psi_{B}^{(N+1)}(\xi+t_\theta(\theta-\xi)),
\end{multline*}
and
\begin{multline*}
F(\theta)=F(\xi)+(\theta-\xi)F'(\xi)+\frac{(\theta-\xi)^{N}}{N!}F^{(N)}(\xi)\\+\frac{(\theta-\xi)^{N+1}}{(N+1)!}F^{(N+1)}(\xi+t_\theta(\theta-\xi)).
\end{multline*}
Thus, going back to the integral $I$ we obtain that if $n$ 
is sufficiently large, then
\begin{multline*}
I  =\exp\left({\rm i}F(\xi)\right)\cdot\int_{\xi-\epsilon}^{\xi+\epsilon} \exp\left({\rm i}\frac{(\theta-\xi)^{N}}{N!}F^{(N)}(\xi)\right)\times \\ \times\exp\left({\rm i}(\theta-\xi)F'(\xi)+{\rm i}\frac{(\theta-\xi)^{N+1}}{(N+1)!}F^{(N+1)}(\xi+t_\theta(\theta-\xi))\right)\,d\theta.
\end{multline*}
Furthermore, 
\begin{multline*}
 \exp\left({\rm i}(\theta-\xi)F'(\xi)+{\rm i}\frac{(\theta-\xi)^{N+1}}{(N+1)!}F^{(N+1)}(\xi+t_\theta(\theta-\xi))\right)\\
 =1+\cO\left(\abs{(\theta-\xi)F'(\xi)}+\frac{|\theta-\xi|^{N+1}}{(N+1)!}\abs{F^{(N+1)}(\xi+t_\theta(\theta-\xi))}\right)\\
  =1+\cO\left(n|\theta-\xi|\left(|f'(\xi)|+|\theta-\xi|^N\right)\right),
\end{multline*}
for $\theta$ in a small neighborhood of $\xi$. This gives
\begin{multline}
\exp\left(-{\rm i}F(\xi)\right)I   
 \\=\int_{\xi-\epsilon}^{\xi+\epsilon}\exp\left({\rm i}\frac{(\theta-\xi)^{N}}{N!}F^{(N)}(\xi)\right)\,d\theta+\cO\left(n\epsilon^2
 |f'(\xi)|+n\epsilon^{N+2}\right).\label{eq:I_1st_decomp}
\end{multline}

Let us verify that with our choice of $\epsilon$ we have 
\begin{itemize}
\item[(i)] $n\epsilon^{2}\abs{f'(\xi)}=o\left(n^{-1/N}\right)$,
\item[(ii)] $n\epsilon^{N+2}=o\left(n^{-1/N}\right)$, 
\item[(iii)] $n\epsilon^{N}\rightarrow\infty$ as $n\rightarrow\infty$.
\end{itemize}
Since $k\le n\psi_{B}'(\xi)<k+D$, we have $n|f'(\xi)|\lesssim 1$. 
Now, (i)--(iii) follow from \eqref{st1}. 

Thus,
\begin{equation}
n\epsilon^2 |f'(\xi)|+n\epsilon^{N+2}=o(n^{-1/N}).\label{eq:O_term_I}
\end{equation}

Next we consider
\begin{multline*}
J =\int_{\xi-\epsilon}^{\xi+\epsilon}\exp\left({\rm i}\frac{(\theta-\xi)^{N}}{N!}F^{(N)}(\xi)\right)\,d\theta\\
 \\=\int_{\xi}^{\xi+\epsilon}\exp\left({\rm i}\frac{(\theta-\xi)^{N}}{N!}F^{(N)}(\xi)\right)\,d\theta+\int_{\xi-\epsilon}^{\xi}\exp\left({\rm i}\frac{(\theta-\xi)^{N}}{N!}F^{(N)}(\xi)\right)\,d\theta\\
 \\=J_{1}+J_{2}.
\end{multline*}
First we estimate $J_1$. Changing the variable 
\[
u=\frac{(\theta-\xi)^{N}}{N!}F^{(N)}(\xi),
\]
we obtain that 

\begin{multline*}
J_{1}  =\frac{1}{N}\left(\frac{N!}{F^{(N)}(\xi)}\right)^{1/N}\int_{0}^{\frac{\epsilon^{N}}{N!}F^{(N)}(\xi)}\frac{\exp\left({\rm i}u\right)}{u^{1-1/N}}\,du\\
 \\=\frac{1}{N}\left(\frac{N!}{F^{(N)}(\xi)}\right)^{1/N}\left(\int_{0}^{\infty}\frac{\exp\left({\rm i}u\right)}{u^{1-1/N}}\,du-\int_{\frac{\epsilon^{N}}{N!}F^{(N)}(\xi)}^{\infty}\frac{\exp\left({\rm i}u\right)}{u^{1-1/N}}\,du\right)\\
 \\ \!=\frac{1}{N}\left(\frac{N!}{F^{(N)}(\xi)}\right)^{1/N}\!\!\left(\exp\left(\frac{{\rm i}\pi}{2N}\right)\Gamma(1/N)+\cO\left(\frac{1}{(\epsilon^{N}F^{(N)}(\xi))^{1-1/N}}\right)\right).
\end{multline*}
Here we use that 
\[
\int_{0}^{\infty}\frac{\exp\left({\rm i}u\right)}{u^{1-1/N}}\,du=\exp\left(\frac{{\rm i}\pi}{2N}\right)\Gamma(1/N),
\]
and (via integration by parts, with $\gamma>0$) that  
\begin{multline*}
-\int_\gamma^{\infty}\frac{\exp\left({\rm i}u\right)}{u^{1-1/N}}\,du =\left[{\rm i}\frac{\exp\left({\rm i}u\right)}{u^{1-1/N}}\right]_\gamma^{\infty}+\left(1-\frac{1}{N}\right){\rm i}\int_{\gamma}^{\infty}\frac{\exp\left({\rm i}u\right)}{u^{2-1/N}}\,du \\
  =\cO\left(\frac{1}{\gamma^{1-1/N}}\right).
\end{multline*}
Here $\gamma=\frac{\epsilon^{N}}{N!}F^{(N)}(\xi)$ and thus $\gamma^{(N-1)/N}  \asymp n^{\delta(N-1)}$.

This gives
\begin{multline*}
J_{1} =\frac{1}{N}\left(\frac{N!}{F^{(N)}(\xi)}\right)^{1/N}\left(\exp\left(\frac{{\rm i}\pi}{2N}\right)\Gamma(1/N)+\cO\left(\frac{1}{n^{\delta(N-1)}}\right)\right)\\
 \\=\frac{1}{N}\left(\frac{N!}{n\psi_{B}^{(N)}(\xi)}\right)^{1/N}\exp\left(\frac{{\rm i}\pi}{2N}\right)\Gamma(1/N)+o\left(\frac{1}{n^{1/N}}\right).
\end{multline*}
If $N$ is even, then $J_2=J_1$, 
\[
J=\frac{2}{N}\left(\frac{N!}{n\psi_{B}^{(N)}(\xi)}\right)^{1/N}\exp\left(\frac{{\rm i}\pi}{2N}\right)\Gamma(1/N)+o\left(\frac{1}{n^{1/N}}\right),
\]
and asymptotic formula \eqref{eq:I_N_even} follows from the above
equality combined with \eqref{eq:I_1st_decomp} and \eqref{eq:O_term_I}.

For odd $N$ we obtain in an analogous way that 
$$
J_2 =\frac{1}{N}\left(\frac{N!}{n\psi_{B}^{(N)}(\xi)}\right)^{1/N}\exp\left(-\frac{{\rm i}\pi}{2N}\right)\Gamma(1/N)+o\left(\frac{1}{n^{1/N}}\right).
$$
Thus, if $N$ is odd, then 
\[
J=\frac{2}{N}\left(\frac{N!}{n\psi_{B}^{(N)}(\xi)}\right)^{1/N}\cos\left(\frac{\pi}{2N}\right)\Gamma(1/N)+o\left(\frac{1}{n^{1/N}}\right),
\]
and asymptotic formula \eqref{eq:I_N_odd} follows from the above
equality combined with \eqref{eq:I_1st_decomp} and \eqref{eq:O_term_I}.

By \eqref{eq:I_N_even} and \eqref{eq:I_N_odd} we obtain that 
$$
\Sigma=\sum_{\ell\in\mathcal L}I_{\ell,d}=\sum_{\ell\in\mathcal L}c_\ell\exp(-{\rm i}d\xi_\ell),\qquad 0\le d<D.
$$
The points $\xi_\ell$ are pairwise disjoint, and hence the square matrix \newline\noindent $\bigl(e^{-{\rm i}d\xi_\ell}\bigr)_{\ell\in\mathcal L,\,0\le d<D}$ is invertible. Since it does not depend on $n$, 
we conclude that 
$$
\Omega=\max_{0\le d<D}\Bigl | \sum_{\ell\in\mathcal L}I_{\ell,d} \Bigr| \gtrsim \sum_{\ell\in\mathcal L}|c_\ell|\asymp n^{-1/N}.
$$
\\
Step 2. Now we deal with 
the $s-D$ integrals 
\[
\int_{\xi_{\ell}-\epsilon_{\ell}}^{\xi_{\ell}+\epsilon_{\ell}}\exp\left({\rm i}n\,f(\theta)\right)\,d\theta, 
\]
$1\le \ell\le s$, $N_\ell\neq N$ or $\psi'_B(\xi_\ell)\neq \psi'_B(\xi_r)$, and choose $\epsilon_{\ell}$ in such a way that the corresponding integrals are $o\left(\frac{1}{n^{1/N}}\right)$. We distinguish
the following two cases.

(1) If $N_{\ell}=N$, then the integral 
\[
\int_{\xi_{\ell}-\epsilon_{\ell}}^{\xi_{\ell}+\epsilon_{\ell}}\exp\left({\rm i}n\,f(\theta)\right)\,d\theta
\]
is treated as follows. By continuity of $\psi_{B}'$ at the point $\xi_{\ell}$
there exists $\eta_{\ell}>0$ (independent of $n$) such that 
\[
|\theta-\xi_{\ell}| \le \eta_{\ell}\implies |\psi_{B}'(\xi_{\ell})-\psi_{B}'(\theta)|\le |\psi_{B}'(\xi_{\ell})-\psi_{B}'(\xi_r)|/2.
\]
We choose $\epsilon_{\ell}=\eta_{\ell}$ and observe that 
\[
n|f'_d(\theta)|\ge n|\psi_{B}'(\xi_{\ell})-\psi_{B}'(\xi_r)|/2-D\gtrsim n,\qquad |\theta-\xi_\ell| \le \epsilon_\ell,\, 0\le d<D.
\]
Moreover, $f=f_d$ is monotonic over the interval $[\xi_{\ell}-\epsilon_{\ell},\xi_{\ell}+\epsilon_{\ell}]$
and $f'$ is monotonic over $[\xi_{\ell}-\epsilon_{\ell},\xi_{\ell}]$ and
$[\xi_{\ell},\xi_{\ell}+\epsilon_{\ell}]$ (because $f''=\psi_{B}''$ vanishes
at $\xi_{\ell}$ and nowhere else on the interval $[\xi_{\ell}-\epsilon_{\ell},\xi_{\ell}+\epsilon_{\ell}]$).
The assumptions of Lemma \ref{VDCP_order_1} are therefore satisfied
and an application of this lemma gives:
\[
\int_{\xi_{\ell}-\epsilon_{\ell}}^{\xi_{\ell}+\epsilon_{\ell}}\exp\left({\rm i}n\,f_d(\theta)\right)\,d\theta=\cO\left(\frac{1}{n}\right).
\]

(2) If $N_{\ell}<N$, then the situation is even simpler: we set $\epsilon_{\ell}=n^{-1/N_{\ell}}$, 
and estimate directly 
\[
\biggl|\int_{\xi_{\ell}-\epsilon_\ell}^{\xi_{\ell}+\epsilon_\ell}\exp\left({\rm i}n\,f_d(\theta)\right)\,d\theta\biggr|\lesssim n^{-1/N_{\ell}}.
\]
\\
Step 3. We need to verify that 
\begin{multline*}
\biggl|\int_{J'_\ell}\exp\left({\rm i}n\,f_d(\theta)\right)\,d\theta\biggr|+\biggl|\int_{J''_\ell}\exp\left({\rm i}n\,f_d(\theta)\right)\,d\theta\biggr|\\
=o(n^{-1/N}),\quad 1\le \ell\le s,\, 0\le d<D.
\end{multline*}
Given $J'_\ell$ (or $J''_\ell$ with an analogous argument), we consider the following cases. 

(i) $\ell\in\mathcal L$. Then $\epsilon_\ell=n^{\delta-(1/N)}$, and considering the Taylor expansion (of order $N-2$) of $\psi_{B}''$
near $\xi_{\ell}$, we obtain for large $n$ and for $\theta$ close to $\xi_{\ell}$ satisfying $|\theta-\xi_\ell|\ge\epsilon_{\ell}$ that 
$$
|\psi_{B}''(\theta)|\gtrsim \epsilon_\ell^{N-2}=n^{\delta(N-2)-(N-2)/N}.
$$
Since $\psi_{B}''$ does not vanish on the
interval $J'_\ell$, we conclude that 
$$
n|f_d''(\theta)|\gtrsim n^{(2/N)+\delta(N-2)},\qquad \theta\in J'_\ell,\, 0\le d<D.
$$
An application of Lemma~\ref{lem:VDCP} gives
\[
\int_{J'_\ell}\exp\left({\rm i}n\,f_d(\theta)\right)\,d\theta=o(n^{-1/N}).
\]

(ii) $\ell\not\in\mathcal L$ and $N_\ell=N$. Since $\epsilon_\ell$ does not depend on $n$, we have 
$$
|\psi_{B}''(\theta)|\gtrsim 1,\qquad \theta\in J'_\ell,
$$
and
$$
n|f_d''(\theta)|\gtrsim n,\qquad \theta\in J'_\ell,\, 0\le d<D.
$$
Therefore, by Lemma~\ref{lem:VDCP}, we have 
\[
\int_{J'_\ell}\exp\left({\rm i}n\,f_d(\theta)\right)\,d\theta=\cO\left(n^{-1/2}\right).
\]

(iii) $N_\ell<N$. Here $\epsilon_\ell=n^{-1/N_\ell}$, and, arguing as above, we conclude that  
$$
n|f_d''(\theta)|\gtrsim n^{2/N_\ell},\qquad \theta\in J'_\ell,\, 0\le d<D,
$$
and
\[
\int_{J'_\ell}\exp\left({\rm i}n\,f_d(\theta)\right)\,d\theta=\cO\left(n^{-1/N_\ell}\right).
\]
\\
Summing up, Steps 1--3 give us that 
\[
\max_{0\le d<D}|\widehat{B^n}(k_d)|\gtrsim n^{-1/N},
\]
which completes the proof. 
\end{proof}

\section{\label{sec:Proof-Theorem2}Proof of Theorem \ref{thm:constr_examples}}

\subsection{The general case $N\protect\ge3$}

\begin{proof}[Proof of Theorem \ref{thm:constr_examples}, formula \eqref{B}.]
Let $N\ge3$ and let $B=\prod_{1\le j\le N}b_{\lambda_j}$ be a finite Blaschke product.  
The argument
$\psi_{B}(\theta)$ of $B(e^{{\rm i}\theta})$ is determined modulo $2\pi$. Furthermore, 
$\psi'_{B}$ is a real analytic $2\pi$-periodic function, 
and by \eqref{eq1}, we have 
$$
\psi'_{B}(\theta)=\Re\sum_{1\le j\le N}\frac{1+\overline{\lambda_{j}}e^{{\rm i}\theta}}{1-\overline{\lambda_{j}}e^{{\rm i}\theta}}.
$$
Since $\psi'_{B}$ is real analytic, 
the function $\psi''_{B}$ has only finite number of zeros on $[0,2\pi]$.
By Theorem~\ref{thm:upper_bd}, to obtain that 
$$
\|\widehat{B^{n}}\|_{\ell^\infty}\asymp n^{-1/N},
$$
it suffices to verify that $\psi''_{B}(0)=\ldots=\psi_{B}^{(N-1)}(0)=0$, 
$\psi_{B}^{(N)}(0)\not=0$, and $\psi''_{B}$ has no zeros of order $N-1$.

Set $u=\exp(\frac{2\pi {\rm i}}{N})$. For $t\in(0,1/(2\pi))$
to be chosen later on, we set $\zeta=t\exp(\frac{\pi {\rm i}(N-1)}{2N})$ and 
\[
\lambda_j=\frac{\overline{\zeta u^j}}{1+\overline{\zeta u^j}}\in\mathbb{D},\qquad1\le j\le N.
\]
Then 
\begin{align*}
\psi'_{B}(\theta) & =-N+2\Re\sum_{1\le j\le N}\frac{1}{1-\overline{\lambda_{j}}e^{{\rm i}\theta}}\\
 & =-N+2\Re\sum_{1\le j\le N}\frac{1+\zeta u^{j}}{1-\zeta u^{j}(e^{{\rm i}\theta}-1)}\\
 & =-N+2\Re\sum_{s\ge0}\sum_{1\le j\le N}(1+\zeta u^{j})\zeta^{s}u^{js}(e^{{\rm i}\theta}-1)^{s}\\
 & =N+2N\Re\sum_{k\ge1}\zeta^{kN}e^{{\rm i}\theta}(e^{{\rm i}\theta}-1)^{kN-1}\\
 & =N+2N\sum_{k\ge1}t^{kN}\Re\Bigl({\rm i}^{k(N-1)}e^{{\rm i}\theta}(e^{{\rm i}\theta}-1)^{kN-1}\Bigr).
\end{align*}

Furthermore, 
\begin{align*}
\psi''_{B}(\theta) & =2N\Re\sum_{k\ge1}\zeta^{kN}{\rm i}e^{{\rm i}\theta}(kNe^{{\rm i}\theta}-1)(e^{{\rm i}\theta}-1)^{kN-2}\\
 & =2N\sum_{k\ge1}t^{kN}\Re\Bigl({\rm i}^{k(N-1)+1}e^{{\rm i}\theta}(kNe^{{\rm i}\theta}-1)(e^{{\rm i}\theta}-1)^{kN-2}\Bigr).
\end{align*}
It is clear that (independently of $t$) we have $\psi''_{B}(0)=\ldots=\psi_{B}^{(N-1)}(0)=0$.
Since ${\rm i}^{2N-2}\in\mathbb{R}$, for some $c=c(t)\not=0$ we have $\psi''_{B}(\theta)\sim c\theta^{N-2}$
at $0$ and, hence, $\psi_{B}^{(N)}(0)\not=0$.

Set 
\begin{align*}
h(\theta) & ={\rm i}^{N}e^{{\rm i}\theta}(Ne^{{\rm i}\theta}-1)(e^{{\rm i}\theta}-1)^{N-2}\\
 & =(-1)^{N-1}e^{{\rm i}\theta(N+2)/2}(N-e^{-{\rm i}\theta})\cdot2^{N-2}(\sin(\theta/2))^{N-2}.
\end{align*}
Then $h(0)=0$ if and only if $\theta\in2\pi\mathbb{Z}$, and 
\[
\frac{h'(\theta)}{h(\theta)}=\frac{N+2}{2}{\rm i}+\frac{{\rm i}e^{-{\rm i}\theta}}{N-e^{-{\rm i}\theta}}+\frac{N-2}{2}\cdot\frac{\cos(\theta/2)}{\sin(\theta/2)}.
\]
Hence, 
\[
\Im\frac{h'(\theta)}{h(\theta)}\ge\frac{N+2}{2}-\frac{1}{N-1}>0,\qquad 0<|\theta|\le \pi.
\]
Thus, at every point $\theta\in]0,2\pi[$ we have 
\[
|\Re h(\theta)|+|\Re h'(\theta)|>0.
\]
Since $\Re h^{(N-2)}(0)=(-1)^{N-1}(N-1)!\not=0$, we have 
\[
\sum_{s=0}^{N-2}|\Re h^{(s)}(\theta)|>0,\qquad |\theta|\le \pi.
\]
By continuity, we can find $\delta>0$ such that 
$$
\sum_{s=0}^{N-2}|\Re h^{(s)}(\theta)|
\ge\delta,\qquad |\theta|\le \pi.
$$

Now we use that 
\begin{multline*}
\psi''_{B}(\theta)=2Nt^{N}\Re h(\theta)\\+2Nt^{2N}\sum_{k\ge2}t^{(k-2)N}\Re\Bigl({\rm i}^{N}e^{{\rm i}\theta}(kNe^{{\rm i}\theta}-1)(e^{{\rm i}\theta}-1)^{kN-2}\Bigr).
\end{multline*}
Therefore, we can fix a small positive $t$, $t<1/(2\pi)$, such that
$$
\sum_{s=2}^{N}|\psi_{B}^{(s)}(\theta)| 
\ge\delta Nt^{N},\qquad |\theta|\le \pi.
$$
Now, $\psi''_{B}$ has no zeros of
order $N-1$. 
Thus, by Theorem~\ref{thm:upper_bd}, for every $N\ge3$, we have constructed a Blaschke product
$B_{N}$ of order $N$ such that $\|\widehat{B_N^n}\|_{\ell^\infty}\asymp n^{-1/N}$. 
\end{proof}

\subsection{The case $N=5$}

In this section we give two explicit examples of finite Blaschke
product $B$ of degree 2 satisfying the estimate 
$$
\|\widehat{B^n}\|_{\ell^\infty}\asymp n^{-1/5}.
$$

In the first example, the zeros of $B$ are of the same modulus, while in the second example they are on the same diameter of the unit disk.

\subsubsection{Example 1}

Let $w\in\mathbb C\setminus\{1\}$ be such that $\Re w>0$,
\begin{equation}
\Re(w)=\Re(w^3)\not=\Re(w^5)
\label{eq-u11}
\end{equation}
and
\begin{equation}
\Re(w^{-1})\not=\Re(w^{-3}).
\label{eq-u21}
\end{equation}
Set $w_1=w$, $w_2=\overline{w}$,
$$
\lambda_j=\frac{w_j-1}{w_j+1},\qquad j=1,2,
$$
and consider 
\[
B=\prod_{1\le j\le 2}b_{\lambda_j}.
\]

For example, we can choose
$$
w=2+{\rm i}.
$$
Then
\begin{gather*}
\Re(w)=\Re(w^3)=2,\\
\Re(w^5)=-38,\quad \Re(w^{-1})=\frac{2}{5},\quad \Re(w^{-3})=\frac{2}{125}.
\end{gather*}

\begin{prop} 
We have
\[
\|\widehat{B^{n}}\|_{\ell^\infty}\asymp n^{-1/5}.
\]
\end{prop}

\begin{proof}
As above, the argument $\psi_{B}(\theta)$ of $B(e^{{\rm i}\theta})$ (determined modulo $2\pi$)
satisfies the relation
\[
\psi'_{B}(\theta)=\Re\sum_{1\le j\le 2}\frac{1+\lambda_je^{{\rm i}\theta}}{1-\lambda_je^{{\rm i}\theta}}.
\]
Furthermore,
\begin{align*}
\psi''_{B}(\theta)&=2\Re\sum_{1\le j\le 2}\frac{{\rm i}\lambda_je^{{\rm i}\theta}}{(1-\lambda_je^{{\rm i}\theta})^2},\\
\psi'''_{B}(\theta)&=-2\Re\sum_{1\le j\le 2}\frac{\lambda_je^{{\rm i}\theta}+\lambda^2_je^{2{\rm i}\theta}}{(1-\lambda_je^{{\rm i}\theta})^3}.
\end{align*}
Set
$$
e^{{\rm i}\theta}=\frac{1-{\rm i}x}{1+{\rm i}x},\qquad x\in\mathbb R,
$$
and define $h(x)=\psi''_{B}(\theta)$. Then $h$ has a zero of order $N$ at $x_0\in\mathbb R$ if and only if $\psi''_{B}$ has a zero of order $N$ at $\theta_0$, $e^{{\rm i}\theta_0}=(1-{\rm i}x_0)/(1+{\rm i}x_0)$. 

By Theorem~\ref{thm:upper_bd}, we need to verify that
$$
h(0)=h'(0)=h''(0)=0,\quad h'''(0)\not=0,
$$
$h$ has no zeros of order $3$ except at the origin, and $\psi'''_{B}(\pi)\not=0$. 

First,
\begin{multline*}
h(x)=2\Re\sum_{1\le j\le 2}\frac{{\rm i}\frac{w_j-1}{w_j+1}\frac{1-{\rm i}x}{1+{\rm i}x}}{(1-\frac{w_j-1}{w_j+1}\frac{1-{\rm i}x}{1+{\rm i}x})^2}=
\frac{1+x^2}2\Re\sum_{1\le j\le 2}\frac{{\rm i}(w^2_j-1)}{(1+{\rm i}w_jx)^2}\\=
\frac{1+x^2}4 \Bigl(\frac{{\rm i}(w^2-1)}{(1+{\rm i}wx)^2}+\frac{{\rm i}(\overline{w^2}-1)}{(1+{\rm i}\overline{w}x)^2}-\frac{{\rm i}(\overline{w^2}-1)}{(1-{\rm i}\overline{w}x)^2}
-\frac{{\rm i}(w^2-1)}{(1-{\rm i}wx)^2} \Bigr)\\=
x(1+x^2)\sum_{1\le j\le 2} \frac{w_j(w^2_j-1)}{(1+w_j^2x^2)^2}.
\end{multline*}
Therefore, $h(0)=0$. Denote
$$
h_1(x)=\sum_{1\le j\le 2} \frac{w_j(w^2_j-1)}{(1+w_j^2x)^2}.
$$
Then $h(x)=x(1+x^2)h_1(x^2)$. We need to verify that
$$
h_1(0)=0,\quad h_1'(0)\not=0.
$$
By \eqref{eq-u11},
$$
h_1(0)=\sum_{1\le j\le 2} (w^3_j-w_j)=0,
$$
and 
$$
h'_1(0)=-2\sum_{1\le j\le 2} (w^5_j-w^3_j)\not=0.
$$

If the function $h$ has a zero of order $3$ at $x_0\not=0$, then $h_1$ has a zero of order $3$ at $y_0=\sqrt{|x_0|}\not=0$. 
We have
$$
h_1(x)= \frac{Q(x)}{\prod_{1\le j\le 2}(1+w_j^2x)^2},
$$
where $Q$ is a polynomial of degree at most $2$, and, hence, $Q$ cannot have a zero of order $3$ at $y_0$.

Finally, by \eqref{eq-u21}, we have 
\begin{multline*}
\psi'''_{B}(\pi)=2\Re\sum_{1\le j\le 2}\frac{\lambda_j-\lambda^2_j}{(1+\lambda_j)^3}=2\Re\sum_{1\le j\le 2}\frac{\frac{w_j-1}{w_j+1}(1-\frac{w_j-1}{w_j+1})}{(1+\frac{w_j-1}{w_j+1})^3}\\=
\frac12\Re\sum_{1\le j\le 2}\frac{w^2_j-1}{w^3_j}=
\Re\bigl(w^{-1}-w^{-3}\bigr)\not=0.
\end{multline*}
\end{proof}

\subsubsection{Example 2}

Here we give another example of a finite Blaschke product
$B$ of degree 2 such that $\|\widehat{B^{n}}\|_{\ell^\infty}\asymp n^{-1/5}$.

Let $w_1,w_2\in(0,\infty)\setminus\{1\}$ be such that 
\begin{equation}
w_1+w_2=w^3_1+w^3_2\not=w^5_1+w^5_2
\label{eq-u12}
\end{equation}
and
\begin{equation}
w^{-1}_1+w^{-1}_2\not=w^{-3}_1+w^{-3}_2.
\label{eq-u22}
\end{equation}
Set  
$$
\lambda_j=\frac{w_j-1}{w_j+1},\qquad j=1,2,
$$
and consider 
\[
B=\prod_{1\le j\le 2}b_{\lambda_j}.
\]

For example, we can choose
$$
w_1=\frac12,\quad w_2=\frac{1+\sqrt{13}}4.
$$
Then
\begin{gather*}
w_1+w_2=w^3_1+w^3_2=\frac{3+\sqrt{13}}4,\quad w^5_1+w^5_2=\frac{63+19\sqrt{13}}{64},\\
 w^{-1}_1+w^{-1}_2=\frac{5+\sqrt{13}}3,\quad w^{-3}_1+w^{-3}_2=\frac{176+16\sqrt{13}}{27}.
\end{gather*}

\begin{prop} 
We have
$$
\|\widehat{B^n}\|_{\ell^\infty}\asymp n^{-1/5}.
$$
\end{prop}

\begin{proof}
As above,
\begin{align*}
\psi'_{B}(\theta)&=\Re\sum_{1\le j\le 2}\frac{1+\lambda_je^{{\rm i}\theta}}{1-\lambda_je^{{\rm i}\theta}},\\
\psi''_{B}(\theta)&=2\Re\sum_{1\le j\le 2}\frac{{\rm i}\lambda_je^{{\rm i}\theta}}{(1-\lambda_je^{{\rm i}\theta})^2},\\
\psi'''_{B}(\theta)&=-2\Re\sum_{1\le j\le 2}\frac{\lambda_je^{{\rm i}\theta}+\lambda^2_je^{2{\rm i}\theta}}{(1-\lambda_je^{{\rm i}\theta})^3}.
\end{align*}
Set
$$
e^{{\rm i}\theta}=\frac{1-{\rm i}x}{1+{\rm i}x},\qquad x\in\mathbb R,
$$
and define $h(x)=\psi''_{B}(\theta)$. 

As above, by Theorem~\ref{thm:upper_bd}, we need to verify that
$$
h(0)=h'(0)=h''(0)=0,\quad h'''(0)\not=0,
$$
$h$ has no zeros of order $3$ except at the origin, and $\psi'''_{B}(\pi)\not=0$. 

First,
\begin{multline*}
h(x)=2\Re\sum_{1\le j\le 2}\frac{{\rm i}\frac{w_j-1}{w_j+1}\frac{1-{\rm i}x}{1+{\rm i}x}}{(1-\frac{w_j-1}{w_j+1}\frac{1-{\rm i}x}{1+{\rm i}x})^2}=
\frac{1+x^2}2\Re\sum_{1\le j\le 2}\frac{{\rm i}(w^2_j-1)}{(1+{\rm i}w_jx)^2}\\=
\frac{1+x^2}4 \Bigl(\frac{{\rm i}(w_1^2-1)}{(1+{\rm i}w_1x)^2}+\frac{{\rm i}(w_2^2-1)}{(1+{\rm i}w_2x)^2}-\frac{{\rm i}(w_1^2-1)}{(1-{\rm i}w_1x)^2}
-\frac{{\rm i}(w_2^2-1)}{(1-{\rm i}w_2x)^2} \Bigr)\\=
x(1+x^2)\sum_{1\le j\le 2} \frac{w_j(w^2_j-1)}{(1+w_j^2x^2)^2}.
\end{multline*}
Therefore, $h(0)=0$. Denote
$$
h_1(x)=\sum_{1\le j\le 2} \frac{w_j(w^2_j-1)}{(1+w_j^2x)^2}.
$$
As above, $h(x)=x(1+x^2)h_1(x^2)$, and we need to verify that
$$
h_1(0)=0,\quad h_1'(0)\not=0.
$$
By \eqref{eq-u12},
$$
h_1(0)=\sum_{1\le j\le 2} (w^3_j-w_j)=0,
$$
and
$$
h'_1(0)=-2\sum_{1\le j\le 2} (w^5_j-w^3_j)\not=0.
$$

If the function $h$ has a zero of order $3$ at $x_0\not=0$, then $h_1$ has a zero of order $3$ at $y_0=\sqrt{|x_0|}\not=0$. 
We have
$$
h_1(x)= \frac{Q(x)}{\prod_{1\le j\le 2}(1+w_j^2x)^2},
$$
where $Q$ is a polynomial of degree at most $2$, and, hence, $Q$ cannot have a zero of order $3$ at $y_0$.

Finally, by \eqref{eq-u22}, we have 
\begin{multline*}
\psi'''_{B}(\pi)=2\sum_{1\le j\le 2}\frac{\lambda_j-\lambda^2_j}{(1+\lambda_j)^3}=2\sum_{1\le j\le 2}\frac{\frac{w_j-1}{w_j+1}(1-\frac{w_j-1}{w_j+1})}{(1+\frac{w_j-1}{w_j+1})^3}\\=
\frac12\sum_{1\le j\le 2}\frac{w^2_j-1}{w^3_j}=
\frac12\sum_{1\le j\le 2}\bigl(w_j^{-1}-w_j^{-3}\bigr)\not=0.
\end{multline*}
\end{proof}

\subsection{The case $N=7$}

Here we give an explicit example of a finite Blaschke product
$B$ of degree 4 such that $\|\widehat{B^n}\|_{\ell^\infty}\asymp n^{-1/7}$.

Let $w_1,w_2\in\mathbb C\setminus\{1\}$ be such that $\Re w_1,\Re w_2>0$,
\begin{equation}
\Re(w_1+w_2)=\Re(w^3_1+w^3_2)=\Re(w_1^5+w_2^5)\not=\Re(w^7_1+w^7_2)
\label{eq-u1}
\end{equation}
and
\begin{equation}
\Re(w^{-1}_1+w^{-1}_2)\not=\Re(w^{-3}_1+w^{-3}_2).
\label{eq-u2}
\end{equation}
Set
\begin{gather*}
w_{j+2}=\overline{w_j},\qquad j=1,2,\\
\lambda_j=\frac{w_j-1}{w_j+1}, \qquad 1\le j\le 4,
\end{gather*}
and consider 
\[
B=\prod_{1\le j\le 4}b_{\lambda_j}.
\]

For example, we can choose
$$
w_1=1+\frac{2{\rm i}}{\sqrt{3}},\quad w_2=2+\frac{{\rm i}}{\sqrt{3}}.
$$
Then
\begin{gather*}
\Re(w_1+w_2)=\Re(w^3_1+w^3_2)=\Re(w_1^5+w_2^5)=3,\\
\Re(w^7_1+w^7_2)=-\frac{421}{9},\quad \Re(w^{-1}_1+w^{-1}_2)=\frac{81}{91},\\ 
\Re(w^{-3}_1+w^{-3}_2)=-\frac{122391}{753571}.
\end{gather*}

\begin{prop}
We have
\[
\|\widehat{B^{n}}\|_{\ell^\infty}\asymp n^{-1/7}.
\]
\end{prop}

\begin{proof}
As above,  
\begin{align*}
\psi'_{B}(\theta)&=\Re\sum_{1\le j\le 4}\frac{1+\lambda_je^{{\rm i}\theta}}{1-\lambda_je^{{\rm i}\theta}},\\
\psi''_{B}(\theta)&=2\Re\sum_{1\le j\le 4}\frac{{\rm i}\lambda_je^{{\rm i}\theta}}{(1-\lambda_je^{{\rm i}\theta})^2},\\
\psi'''_{B}(\theta)&=-2\Re\sum_{1\le j\le 4}\frac{\lambda_je^{{\rm i}\theta}+\lambda^2_je^{2{\rm i}\theta}}{(1-\lambda_je^{{\rm i}\theta})^3}.
\end{align*}
Set
$$
e^{{\rm i}\theta}=\frac{1-{\rm i}x}{1+{\rm i}x},\qquad x\in\mathbb R,
$$
and define $h(x)=\psi''_{B}(\theta)$. Then $h$ has a zero of order $N$ at $x_0\in\mathbb R$ if and only if $\psi''_{B}$ has a zero of order $N$ at $\theta_0$, $e^{{\rm i}\theta_0}=(1-{\rm i}x_0)/(1+{\rm i}x_0)$. 

We need to verify that
$$
h(0)=h'(0)=h''(0)=h'''(0)=h^{(4)}(0)=0,\quad h^{(5)}(0)\not=0,
$$
$h$ has no zeros of order $5$ except at the origin, and $\psi'''_{B}(\pi)\not=0$. 

As above,
\begin{multline*}
h(x)=2\Re\sum_{1\le j\le 4}\frac{{\rm i}\frac{w_j-1}{w_j+1}\frac{1-{\rm i}x}{1+{\rm i}x}}{(1-\frac{w_j-1}{w_j+1}\frac{1-{\rm i}x}{1+{\rm i}x})^2}=
\frac{1+x^2}2\Re\sum_{1\le j\le 4}\frac{{\rm i}(w^2_j-1)}{(1+{\rm i}w_jx)^2}\\=
\frac{1+x^2}4\sum_{1\le j\le 2}\Bigl(\frac{{\rm i}(w^2_j-1)}{(1+{\rm i}w_jx)^2}+\frac{{\rm i}(\overline{w^2_j}-1)}{(1+{\rm i}\overline{w_j}x)^2}-\frac{{\rm i}(\overline{w^2_j}-1)}{(1-{\rm i}\overline{w_j}x)^2}
-\frac{{\rm i}(w^2_j-1)}{(1-{\rm i}w_jx)^2} \Bigr)\\=
x(1+x^2)\sum_{1\le j\le 4} \frac{w_j(w^2_j-1)}{(1+w_j^2x^2)^2}.
\end{multline*}
Therefore, $h(0)=0$. Denote
$$
h_1(x)=\sum_{1\le j\le 4} \frac{w_j(w^2_j-1)}{(1+w_j^2x)^2}.
$$
Then $h(x)=x(1+x^2)h_1(x^2)$. We need to verify that
$$
h_1(0)=h_1'(0)=0,\quad h_1''(0)\not=0.
$$
By \eqref{eq-u1},
$$
h_1(0)=\sum_{1\le j\le 4} (w^3_j-w_j)=0.
$$
Furthermore,
$$
h'_1(x)=-2\sum_{1\le j\le 4} \frac{w^3_j(w^2_j-1)}{(1+w_j^2x)^3}.
$$
Again by \eqref{eq-u1},
$$
h'_1(0)=-2\sum_{1\le j\le 4} (w^5_j-w^3_j)=0.
$$
Next,
$$
h''_1(x)=6\sum_{1\le j\le 4} \frac{w^5_j(w^2_j-1)}{(1+w_j^2x)^4},
$$
and by \eqref{eq-u2},
$$
h''_1(0)=6\sum_{1\le j\le 4} w^5_j(w^2_j-1)\not=0.
$$

If the function $h$ has a zero of order $5$ at $x_0\not=0$, then $h_1$ has a zero of order $5$ at $y_0=\sqrt{|x_0|}\not=0$. 
We have
$$
h_1(x)= \frac{Q(x)}{\prod_{1\le j\le 4}(1+w_j^2x)^2},
$$
where $Q$ is a polynomial of degree at most $6$. Since $h_1(0)=h_1'(0)=0$, $Q$ cannot have a zero of order $5$ at $y_0$.

Finally, by \eqref{eq-u2}, we have 
\begin{multline*}
\psi'''_{B}(\pi)=2\Re\sum_{1\le j\le 4}\frac{\lambda_j-\lambda^2_j}{(1+\lambda_j)^3}=2\Re\sum_{1\le j\le 4}\frac{\frac{w_j-1}{w_j+1}(1-\frac{w_j-1}{w_j+1})}{(1+\frac{w_j-1}{w_j+1})^3}\\=
\frac12\Re\sum_{1\le j\le 4}\frac{w^2_j-1}{w^3_j}=
\Re\sum_{1\le j\le 2}\bigl(w^{-1}_j-w^{-3}_j\bigr)\not=0.
\end{multline*}
\end{proof}

\section{\label{Proof-Theorem3} Proof of Theorem \ref{prop_phi} }

We recall that 
$$
e_{nm} =  \frac{(1 - \abs{\lambda_m}^2)^{1/2}}{z-\lambda_m} B^n.
$$
We set $g=S^{-1}e_{nm}$ so that $g = g(0)+ z e_{nm}$. Then 
\begin{equation} \label{etoile}
\|S^{-1}\|_{\ell^\infty_A\to\ell^\infty_A} \ge  \frac{\|\widehat{S^{-1}e_{nm}}\|_{\ell^\infty}}
{\|\widehat{e_{nm}}\|_{\ell^\infty}} 
= \frac{\|\widehat{g}\|_{\ell^\infty}}{\|\widehat{e_{nm}}\|_{\ell^\infty}}
\ge \frac{|g(0)|}{\|\widehat{e_{nm}}\|_{\ell^\infty}}.
\end{equation}
Let us first concentrate on estimating $\|\widehat{e_{nm}}\|_{\ell^\infty}$. We follow the same steps as in the proof of Theorem~\ref{thm:upper_bd}. We set $d_{nm}= \frac{e_{nm}}{(1-\abs{\lambda_m}^2)^{1/2}}$. We have 
$$
\widehat{d_{nm}}(k) = \frac{1}{2\pi} \int_0^{2\pi} \frac{B^n(e^{{\rm i}\theta})}{e^{{\rm i}\theta} - \lambda_m} e^{-{\rm i}k\theta} \,d\theta 
 = \frac{1}{2\pi} \int_0^{2\pi} \frac{e^{{\rm i}n\,f(\theta)}}{e^{{\rm i}\theta} - \lambda_m}  \,d\theta,
$$
where $f(\theta) = \psi(\theta) - \frac{k}{n} \theta$ and $\psi(\theta)=\arg(B(e^{{\rm i}\theta}))$. Using the same notation as in 
Theorem~\ref{thm:upper_bd}, we consider the sequence $(\xi_{\ell})_{\ell=1}^s$ of consecutive zeros of $\psi''$ on $[0,2\pi)$ with respective multiplicities $(N_\ell-2)_{\ell=1}^s$, $N_\ell \ge 3$. 
We have: 
\begin{multline*}
\widehat{d_{nm}}(k)  =\sum_{\ell=1}^{s}\int_{\xi_{\ell}-\epsilon}^{\xi_{\ell}+\epsilon}\frac{e^{{\rm i}n\,f(\theta)}}{e^{{\rm i}\theta} - \lambda_m}\,d\theta
  +\sum_{\ell=1}^{s-1}\int_{\xi_{\ell}+\epsilon}^{\xi_{\ell+1}-\epsilon}\frac{e^{{\rm i}n\,f(\theta)}}{e^{{\rm i}\theta} - \lambda_m}\,d\theta\\
  +\int_{0}^{\xi_{1}-\epsilon}\frac{e^{{\rm i}n\,f(\theta)}}{e^{{\rm i}\theta} - \lambda_m}\,d\theta
  +\int_{\xi_{s}+\epsilon}^{2\pi}\frac{e^{{\rm i}n\,f(\theta)}}{e^{{\rm i}\theta} - \lambda_m}\,d\theta.
\end{multline*}
From now on we put 
\[
\epsilon=n^{-1/N}\rightarrow0,\qquad n\rightarrow\infty,
\]
where $N = \max_{1 \le \ell \le s} N_\ell$.
Clearly 
\[
\Abs{\int_{\xi_{\ell}-\epsilon}^{\xi_{\ell}+\epsilon}\frac{e^{{\rm i}n\,f(\theta)}}{e^{{\rm i}\theta} - \lambda_m}\,d\theta}\le \frac{2\epsilon}
{\bigl|1-|\lambda_m|\bigr|},
\]
and therefore the first sum is bounded from above as 
\[
\Abs{\sum_{\ell=1}^{s}\int_{\xi_{\ell}-\epsilon}^{\xi_{\ell}+\epsilon}\frac{e^{{\rm i}n\,f(\theta)}}{e^{{\rm i}\theta} - \lambda_m}\,d\theta}\lesssim n^{-1/N}.
\]
Furthermore, $f''=\psi''$ does not vanish on the intervals $(0,\xi_{1}-\epsilon)$, 
$(\xi_{s}+\epsilon,2\pi)$ and $(\xi_{\ell}+\epsilon,\xi_{\ell+1}-\epsilon)$
for $\ell=1,\ldots, s-1$. Writing the Taylor expansion (of order $N_{\ell}-2$)
of $\psi''$ near $\xi_{\ell}$ we find 
\[
\psi''(\theta)=\frac{\psi^{(N_{\ell})}(\xi_{\ell})}{(N_{\ell}-2)!}\left(\theta-\xi_{\ell}\right)^{N_{\ell}-2}\left(1+\cO\left(\theta-\xi_{\ell}\right)\right),
\]
and it follows that if $\theta$ is close to $\xi_\ell$ but satisfies $\abs{\theta - \xi_\ell} \ge \varepsilon$, then
\[ \abs{\psi''_B(\theta)} \gtrsim n^{-(N_\ell-2)/N},\]
for $n$ large enough. Below we apply Lemma \ref{lem:VDCP} to the remaining integrals:
\begin{equation} \label{integral}
\int_{\xi_{l}+\epsilon}^{\xi_{l+1}-\epsilon} \frac{e^{{\rm i}n\,f(\theta)}}{e^{{\rm i}\theta}-\lambda_m}\,d\theta,\quad 
\int_{0}^{\xi_{1}-\epsilon} \frac{e^{{\rm i}n\,f(\theta)}}{e^{{\rm i}\theta}-\lambda_m}\,d\theta,\quad 
\int_{\xi_{s}+\epsilon}^{2\pi}\frac{e^{{\rm i}n\,f(\theta)}}{e^{{\rm i}\theta}-\lambda_m}\,d\theta.
\end{equation}
We describe how to get an upper estimate on the first integral in \eqref{integral}, the argument for two others being similar. We have 
\begin{multline*} 
\biggl|\int_{\xi_{l}+\epsilon}^{\xi_{l+1}-\epsilon} \frac{e^{{\rm i}n\,f(\theta)}}{e^{{\rm i}\theta}-\lambda_m}\,d\theta\biggr|\\ \le
\biggl|\int_{\xi_{l}+\epsilon}^{\xi_{l+1}-\epsilon} \Re\biggl(\frac{e^{{\rm i}n\,f(\theta)}}{e^{{\rm i}\theta}-\lambda_m}\biggr)\,d\theta\biggr|+
\biggl|\int_{\xi_{l}+\epsilon}^{\xi_{l+1}-\epsilon} \Im\biggl(\frac{e^{{\rm i}n\,f(\theta)}}{e^{{\rm i}\theta}-\lambda_m}\biggr)\,d\theta\biggr|.
\end{multline*} 
We set $\lambda_m=\rho_m\exp({\rm i}\theta_m)$ and 
\[ 
G(\theta)=\Re \left(\frac{1}{e^{{\rm i}\theta}-\lambda_m}\right) = \frac{\cos{\theta}-\rho_m \cos \theta_m}{(\cos{\theta}-\rho_m \cos \theta_m)^2+(\sin{\theta}-\rho_m \sin \theta_m)^2}.
\]
The function $G$ is bounded on $[0, 2\pi]$ and it vanishes twice on $[0, 2\pi)$. 
Furthermore, $G'$ has a finite number of zeros on $[0, 2\pi)$. Therefore, we can split every interval of integration into a (uniformly bounded in $n$) number of intervals so that $G$ is of constant sign and monotonic on each of them. Then 
\[
\abs{F''}\gtrsim n^{2/N}
\]
on these intervals and therefore, applying Lemma~\ref{lem:VDCP}, we obtain that the corresponding integrals are $\cO\left(n^{-1/N}\right)$. Applying the same steps with $G(\theta)= \Bigl(\dfrac{1}{e^{{\rm i}\theta}-\lambda_m}\Bigr)$, 
and estimating in the same way the remaining integrals in \eqref{integral}, we conclude that $\|\widehat{e_{nm}}\|_{\ell^\infty} \lesssim n^{-1/N}$. 
\\ 
Now we return to (\ref{etoile}). It remains to estimate $|g(0)|$.
We use that $g =S^{-1}e_{nm} \in K_B$, and that  
$$
K_B = \Biggl\{ \frac{P(z)}{\prod\limits_{i=1}^m (1-\bar{\lambda}_i z)^n}: \deg P \le nm - 1\Biggr\}.
$$  
Thus we can determine $g(0)$ by the relation  
\begin{equation} \label{g(0)}
 g(0) + (1-|\lambda_m|^2)^{1/2}  \frac{z(z-\lambda_m)^{n-1}}{(1-\bar{\lambda}_m z)^n} \prod_{j=1}^{m-1} \left(\frac{z-\lambda_j}{1-\bar{\lambda}_j z}\right)^n  \in K_B.
\end{equation}
Since
$$
\frac{z-\lambda}{1-\bar{\lambda}z} = -\frac1{\bar{\lambda}}+ \frac{(1/\bar{\lambda})-\lambda}{1-\bar{\lambda}z},\qquad 
\frac{z}{1-\bar{\lambda}z} = -\frac1{\bar{\lambda}}+ \frac{1/\bar{\lambda}}{1-\bar{\lambda}z},
$$
relation \eqref{g(0)} implies that
\[ 
g(0) + (1-|\lambda_m|^2)^{1/2}\prod_{j=1}^m \frac{1}{(-\bar{\lambda}_j)^n} + \frac{P(z)}{\prod_{j=1}^m (1 - \bar{\lambda}_j z)^n} \in K_B,
\]
where $\deg P \le  nm-1$. Therefore, the condition $g \in K_B$ implies that 
$$
g(0) = - (1-|\lambda_m|^2)^{1/2}\prod_{j=1}^m \frac{1}{(-\bar{\lambda}_j)^n}.
$$
Thus, \eqref{etoile} yields
\begin{equation} 
\label {result th3}
\|S^{-1}\|_{\ell^\infty_A\to\ell^\infty_A} \gtrsim \frac{n^{1/N}}{\prod_{j=1}^m |\lambda_j|^n},
\end{equation}
which completes the proof of part (i) because of \eqref{det1}.\\

To prove part (ii), we use the analytic expression of $\Phi$ established by Gluskin--Meyer--Pajor \cite{GMP}: 
\[
\Phi(\lambda_1,\dots,\lambda_1,\dots,\lambda_m,\dots,\lambda_m) = \sup \{ |\det T|\cdot \|T^{-1}\|  \}, 
\]
where we take the supremum by all norms in $\mathbb C^{nm}$ and all $T$ such that $\|T\|\le 1$ in the induced norm and 
$\sigma_T=\{\lambda_1, \ldots, \lambda_m\}$ with multiplicity $n$ of every eigenvalue. 
We conclude by setting $T=S$ and using \eqref{det1} and \eqref{result th3}. 
\qed

\begin{remark}\label{rem9} Here we establish part (ii) of Theorem~\ref{prop_phi} as a direct application of Theorem \ref{thm:upper_bd}.
To this aim we
reproduce and adapt the duality method used to prove the main result of \cite{SZ2}.
We recall the definition of the Wiener algebra, which is the subset
of $H(\mathbb{D})$ of absolutely convergent Fourier series,
\begin{align*}
W:=\Bigl\{f=\sum_{k\ge0}\hat{f}(k)z^{k}:\|f\|_W:=\sum_{k\ge0}\abs{\hat{f}(k)}<\infty\Bigr\}.
\end{align*}
Let $B$ be the finite Blaschke product with simple zeros at $\lambda_{1},\dots,\lambda_{m}$.
It is easily verified that if $f$ is in $W$ and $f(\lambda)=0$
with $\lambda\in\mathbb{D}$, then 
\[
\Bigl\|\frac{f}{z-\lambda}\Bigr\|_W\le\frac{\|f\|_W}{1-\abs{\lambda}}.
\]
In particular, $\frac{f}{z-\lambda}$ is also in $W$. 
Therefore, for any $h\in W$
with 
\[
h(\lambda_j)=h'(\lambda_j)=\ldots=h^{(n-1)}(\lambda_j)=0,\qquad j=1,\dots,m,
\]
we have $h/B^n\in W$.
With this observation we rewrite the Gluskin--Meyer--Pajor expression
for $\Phi$ as follows: 
\begin{multline}\label{star2}
\Phi(\lambda_1,\dots,\lambda_1,\dots,\lambda_m,\dots,\lambda_m)\\
=\inf\biggl\{ \|h\|_W-\abs{h(0)} : h\in B^{n}W,\, h(0)=\prod_{j=1}^{m}\lambda_j^n\biggr\}.
\end{multline}
Let
$L^{2}(\partial\mathbb{D})$ be the usual $L^{2}$ space on the unit circle 
$\partial\mathbb{D}$, equipped with the standard scalar product 
\[
\left\langle f,g\right\rangle :=\int_{-\pi}^{\pi}f(e^{{\rm i}\varphi})\overline{g(e^{{\rm i}\varphi})}\,\frac{d\varphi}{2\pi}.
\]
For $f=\sum_{k}\hat{f}(k)z^{k}$, $g=\sum_{k}\hat{g}(k)z^{k}\in H^2$, 
it is well known \cite{MR} that the $L^{2}(\partial\mathbb{D})$
scalar product can be written as 
\[
\left\langle f,\,g\right\rangle =\sum_{j\ge0}\hat{f}(j)\overline{\hat{g}(j)}.
\]
Let $h=B^ng$ with $g\in W$ and $h(0)=\prod_{j=1}^m\lambda_j^n$. 
Then
$$
\|h\|_W\cdot\|B^n\|_{\ell^\infty_A}\ge \bigl|\bigl\langle h,B^{n}\bigr\rangle\bigr|  =\bigl|\bigl\langle g,1\bigr\rangle\bigr|
  =|g(0)|=\Bigl|\frac{h(0)}{B^{n}(0)}\Bigr|=1.
$$
It follows that any candidate function $h$ in \eqref{star2}
satisfies 
\[
\|h\|_W\ge\frac{1}{\|B^{n}\|_{\ell^\infty_A}}, 
\]
and consequently 
\[
\Phi\left(\lambda_1,\dots,\lambda_1,\dots,\lambda_m,\dots,\lambda_m\right)\ge\frac{1}{\|B^n\|_{\ell^\infty_A}}-\prod_{j=1}^{m}
\abs{\lambda_j}^n.
\]
It remains to apply Theorem \ref{thm:upper_bd} to obtain that 
\[
\Phi(\lambda_1, \dots, \lambda_1, \dots, \lambda_m,\dots, \lambda_m) \gtrsim n^{1/N},\qquad n\to\infty.
\]
\end{remark}

\end{document}